\documentclass[a4paper,11pt]{amsart}
\usepackage{amsfonts}
\usepackage{amssymb}
\usepackage[utf8]{inputenc}
\usepackage{amsmath}
\usepackage{pdflscape}
\usepackage{graphicx}
\providecommand{\U}[1]{\protect\rule{.1in}{.1in}}
\usepackage[colorlinks=true, linkcolor=red, citecolor=blue]{hyperref}

\usepackage[]{epsfig}
\usepackage[]{pstricks}
\usepackage{tikz}
\newtheorem{theorem}{Theorem}[section]
\newtheorem{lemma}{Lemma}[section]
\newtheorem{definition}{Definition}[section]
\newtheorem{remark}{Remark}[section]

\newtheorem{proposition}{Proposition}[section]

\numberwithin{equation}{section}
\newcommand{\R}{\mathbb R}

\newcommand{\N}{{\mathbb N}}

\newcommand{\G}{{\mathcal{G}}}

\newcommand{\re}{{\text{Re }}}

\newcommand{\px}{\partial_x}
\newcommand{\pt}{\partial_t}

\def\norm#1{\|#1\|}

\providecommand{\norm}[1]{\left\lVert#1\right\rVert}
\numberwithin{equation}{section}

\begin{document}

	\pagenumbering{arabic}	
\title[Control biharmonic: Star graphs]{Controllability for Schrödinger type system with mixed dispersion on compact star graphs}

\author[Capistrano--Filho]{Roberto de A. Capistrano--Filho}
\address{\emph{Departamento de Matem\'atica,  Universidade Federal de Pernambuco (UFPE), 50740-545, Recife (PE), Brazil.}}
\email{roberto.capistranofilho@ufpe.br}
\author[Cavalcante]{M\'arcio Cavalcante}
\address{\emph{Instituto de Matem\'{a}tica, Universidade Federal de Alagoas (UFAL), Macei\'o (AL), Brazil}}
\email{marcio.melo@im.ufal.br}
\author[Gallego]{Fernando A. Gallego}
\address{\emph{Departamento de Matematicas y Estad\'istica, Universidad Nacional de Colombia (UNAL), Cra 27 No. 64-60, 170003, Manizales, Colombia}}
\email{fagallegor@unal.edu.co}

\makeatletter
\@namedef{subjclassname@2020}{\textup{2020} Mathematics Subject Classification}
\makeatother

\begin{abstract}
In this work we are concerned with solutions to the linear Schrödinger type system with mixed dispersion, the so-called biharmonic Schrödinger equation. Precisely, we are able to prove an exact control property for these solutions with the control in the energy space posed on an oriented star graph structure $\mathcal{G}$ for $T>T_{min}$, with $$T_{min}=\sqrt{ \frac{ \overline{L} (L^2+\pi^2)}{\pi^2\varepsilon(1-  \overline{L} \varepsilon)}},$$ when the couplings and the controls appear only on the Neumann boundary conditions.
\end{abstract}
\maketitle

\section{Introduction}
The fourth-order nonlinear Schr\"odinger (4NLS) equation or biharmonic cubic nonlinear \linebreak Schr\"odinger equation
\begin{equation}
\label{fourtha}
i\partial_tu +\partial_x^2u-\partial_x^4u=\lambda |u|^2u,
\end{equation}
have been introduced by Karpman \cite{Karpman} and Karpman and Shagalov \cite{KarSha} to take into account the role of small fourth-order dispersion terms in the propagation of intense laser beams in a bulk medium with Kerr nonlinearity. Equation \eqref{fourtha} arises in many scientific fields such as quantum mechanics, nonlinear optics and plasma physics, and has been intensively studied with fruitful references (see \cite{Ben,Karpman} and references therein).

In the past twenty years such 4NLS have been deeply studied from different mathematical points of view. For example, Fibich \textit{et al.} \cite{FiIlPa} worked various properties of the equation in the subcritical regime, with part of their analysis relying on very interesting numerical developments. The well-posedness problem and existence of the solutions has been shown (see, for instance, \cite{tsutsumi}) by means of the energy method, harmonic analysis, etc.
\subsection{Dispersive models on star graphs}
The study of nonlinear dispersive models in a metric graph has attracted a lot of attention of mathematicians, physicists, chemists and engineers,  see for details \cite{BK, BlaExn08, BurCas01, Mug15} and references therein. In particular, the framework prototype  (graph-geometry)  for description of these phenomena have been a {\it star graph} $\mathcal G$, namely, on metric graphs with $N$  half-lines of the form $(0, +\infty)$  connecting at a common vertex $\nu=0$, together with a nonlinear equation suitably defined on the edges such as the nonlinear Schr\"odinger equation (see Adami {\it{et al.}} \cite{AdaNoj14, AdaNoj15} and Angulo and Goloshchapova \cite{AngGol17a, AngGol17b}). We note that with the introduction of nonlinearities in the dispersive models, the network provides a nice field, where one can look for interesting soliton propagation and nonlinear dynamics in general.   A central point that makes this analysis a delicate problem is the presence of a vertex where the underlying one-dimensional star graph should bifurcate (or multi-bifurcate in a general metric graph).

Looking at other nonlinear dispersive systems on graph structure, we have some interesting results. For example, related with well-posedness theory, the second author in \cite{Cav}, studied the local well-posedness for the Cauchy problem associated to Korteweg-de Vries equation in a metric star graph with three semi-infinite edges given by one negative half-line and two positives half-lines attached to a common vertex $\nu=0$ (the $\mathcal Y$-junction framework).  Another nonlinear dispersive equation, the Benjamin--Bona--Mahony (BBM) equation, is treated in \cite{bona,Mugnolobbm}. More precisely, Bona and Cascaval \cite{bona} obtained local well-posedness in Sobolev space $H^1$ and Mugnolo and Rault \cite{Mugnolobbm} showed the existence of traveling waves for the BBM equation on graphs. Using a different approach Ammari  and Crépeau  \cite{AmCr1} derived results of well-posedness and, also, stabilization for the Benjamin-Bona-Mahony equation in a star-shaped network with bounded edges.

Recently, in \cite{CaCaGa1},  the authors deals to present answers for some questions left in \cite{CaCaGa} concerning the study of the cubic fourth order Schr\"odinger equation in a star graph structure $\mathcal{G}$. Precisely, they considered $\mathcal{G}$ composed by $N$ edges parameterized by half-lines $(0,+\infty)$ attached with a common vertex $\nu$. With this structure the manuscript studied the well-posedness of a dispersive model on star graphs with three appropriate vertex conditions.

Regarding to the control theory and inverse problems, let us cite some previous works on star graphs. Ignat \textit{et al.} in \cite{Ignat2011} worked on the inverse problem for the heat equation and the Schr\"odinger equation on a tree. Later on, Baudouin and Yamamoto  \cite{Baudouin} proposed a unified - and simpler - method to study the inverse problem of determining a coefficient. Results of stabilization and boundary controllability for KdV equation on star-shaped graphs was also proved in \cite{AmCr,Cerpa,Cerpa1}. Finally, recently, Duca in \cite{Duca,Duca1} showed the controllability of the bilinear Schrödinger equation defined on a compact graph. In booth works, with different main goals, the author showed control properties for this system. 

We caution that this is only a small sample of the extant work on graphs structure for partial differential equations.

\subsection{Functional framework} 
Let us define the graphs $\G$ given by the central node  $0$ and  edges $I_j$, for  $j=1,2,\cdots, N$.  Thus, for any function $f:	\G\rightarrow \mathbb C$, we set $f_j= f|_{I_j},$
$$
L^2(\G):= \bigoplus_{j=1}^{N}L^2(I_j):=\left\lbrace f: \G \rightarrow \R: f_j \in L^2(I_j), j \in \{1,2,\cdot, N\} \right\rbrace, \quad \|f\|_2=\left( \sum_{j=1}^N \|f_j\|_{L^2(I_j)}\right)^{1/2}
$$
and 
$$
\left( f, g\right)_{L^2(\G)}:= \re \int_{-l_1}^0   f_1(x)\overline{g_1(x)}dx +\re \sum_{j=2}^{N}\int_{0}^{l_j}   f_j(x)\overline{g_j(x)}dx.
$$
Also, we need the following spaces
$$
H^m_0(\G):= \bigoplus_{j=1}^{N}H^m_0(I_j):=\left\lbrace f: \G \rightarrow \mathbb C: f_j \in H^m(I_j), j \in \{1,2,..., N\}\right\rbrace, 
$$
where $\partial_x^j f_1(-l_1)=\partial_x^jf_j(l_j)=0, \ j \in \{1,2,..., m-1\}$ and $ f_1(0)=\alpha_j f_j(0),  j \in \{1,2,..., N\}$, with
$$
 \|f\|_{H^m_0(\G)}=\left( \sum_{j=1}^{N} \|f_j\|_{H^m(I_j)}^2\right)^{1/2},
$$
for $m\in \N$ with the natural inner product of  $H^s_0(I_j)$. We often write,
$$
\int_{\mathcal{G}} f d x=\int_{-l_{1}}^{0} f_{1}(x) d x+\sum_{j=2}^{N} \int_{0}^{l_{j}} f_{j}(x) d x
$$
Then the inner products and the norms of the Hilbert spaces $L^2(\G)$ and $H^m_0 (\G)$ are defined by 
\begin{align*}
\langle f, g\rangle_{L^{2}(\mathcal{G})}=\re\int_{\mathcal{G}} f(x) \overline{ g(x)} d x\quad \text { and } \quad \|f\|_{L^{2}(\mathcal{G})}^{2} &=\int_{\mathcal{G}}|f|^{2} d x,
\\
 \langle f, g\rangle_{H_{0}^{m}(\mathcal{G})}=\re \sum_{k\leq m}\int_{\mathcal{G}} \partial^k_x f_{x}(x) \overline{ \partial_x^k g_{x}}(x) d x \quad \text { and } \quad \|f\|_{H_{0}^{m}(\mathcal{G})}^{2} &=\sum_{k\leq m}\int_{\mathcal{G}}\left|\partial_x^k f\right|^{2} d x.
\end{align*}
We will denote $H^{-s}(\G)$ the dual of $H^s_0(\G)$. By using Poincaré inequality, it follows that 
$$
\|f\|_{L^2(\G)}^2 \leq \frac{L^2}{\pi^2} \|\partial_x f\|_{L^2(\G)}^2, \quad \forall f  \in H^1_0(\G),
$$
where $L=\max_{j=1,2,...,N}\left\lbrace l_j \right\rbrace.$Thus, we have that 
\begin{equation}\label{poincare}
 \sum_{k=1 }^m  \|\partial_x^k f\|_{L^2(\G)}^2 \leq \|f\|_{H_{0}^{m}(\mathcal{G})}^{2} \leq  \left(  \frac{L^2}{\pi^2}+1\right)\sum_{k=1 }^m  \|\partial_x^k f\|_{L^2(\G)}^2.
\end{equation}

\subsection{Setting of the problem and main result} Let us now present the problem that we will study in this manuscript. Due the results presented in  work \cite{CaCaGa1}, naturally, we should see what happens for the control properties for a linear Schrödinger type system with mixed dispersion on compact graph structure $\mathcal{G}$ of $(N+1)$ edges $e_{j}$ (where $\left.N \in \mathbb{N}^{*}\right)$, of lengths $l_{j}>0, j \in\{1, . ., N+1\}$, connected at one vertex that we assume to be 0 for all the edges. Precisely, we assume that the first edge $e_{1}$ is parametrized on the interval $I_{1}:=\left(-l_{1}, 0\right)$ and the $N$ other edges $e_{j}$ are parametrized on the interval $I_{j}:=\left(0, l_{j}\right)$. On each edge we pose a linear biharmonic NLS equation $(Bi-NLS)$. On the first edge $(j=1)$ we put no control and on the other edges $(j=2, \cdots, N+1)$ we consider Neumann boundary controls (see Fig. \ref{control}).

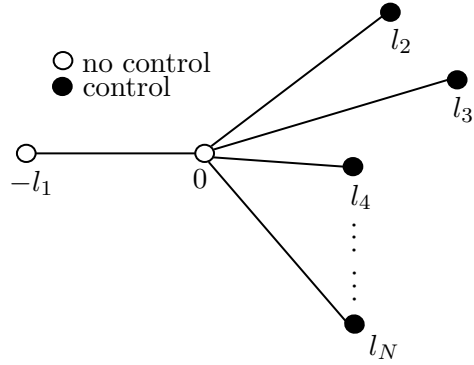
\begin{figure}[h!]
		\centering	
\tikzset{every picture/.style={line width=0.75pt}} 
\begin{tikzpicture}[x=0.50pt,y=0.50pt,yscale=-1,xscale=1]
\draw    (149.5,141) -- (269,141) ;
\draw   (269,141) .. controls (269,137.41) and (272.13,134.5) .. (276,134.5) .. controls (279.87,134.5) and (283,137.41) .. (283,141) .. controls (283,144.59) and (279.87,147.5) .. (276,147.5) .. controls (272.13,147.5) and (269,144.59) .. (269,141) -- cycle ;
\draw   (135.5,141) .. controls (135.5,137.41) and (138.63,134.5) .. (142.5,134.5) .. controls (146.37,134.5) and (149.5,137.41) .. (149.5,141) .. controls (149.5,144.59) and (146.37,147.5) .. (142.5,147.5) .. controls (138.63,147.5) and (135.5,144.59) .. (135.5,141) -- cycle ;
\draw    (281,135.5) -- (409,37.5) ;
\draw  [fill={rgb, 255:red, 0; green, 0; blue, 0 }  ,fill opacity=1 ] (408,34.5) .. controls (408,30.91) and (411.13,28) .. (415,28) .. controls (418.87,28) and (422,30.91) .. (422,34.5) .. controls (422,38.09) and (418.87,41) .. (415,41) .. controls (411.13,41) and (408,38.09) .. (408,34.5) -- cycle ;
\draw    (283,138) -- (458,85.5) ;
\draw    (279,147.5) -- (381.5,267.5) ;
\draw    (282.5,144) -- (380,151) ;
\draw  [fill={rgb, 255:red, 0; green, 0; blue, 0 }  ,fill opacity=1 ] (458,85.5) .. controls (458,81.91) and (461.13,79) .. (465,79) .. controls (468.87,79) and (472,81.91) .. (472,85.5) .. controls (472,89.09) and (468.87,92) .. (465,92) .. controls (461.13,92) and (458,89.09) .. (458,85.5) -- cycle ;
\draw  [fill={rgb, 255:red, 0; green, 0; blue, 0 }  ,fill opacity=1 ] (380,151) .. controls (380,147.41) and (383.13,144.5) .. (387,144.5) .. controls (390.87,144.5) and (394,147.41) .. (394,151) .. controls (394,154.59) and (390.87,157.5) .. (387,157.5) .. controls (383.13,157.5) and (380,154.59) .. (380,151) -- cycle ;
\draw  [fill={rgb, 255:red, 0; green, 0; blue, 0 }  ,fill opacity=1 ] (381.5,267.5) .. controls (382.68,264.11) and (386.59,262.39) .. (390.24,263.65) .. controls (393.9,264.92) and (395.9,268.7) .. (394.73,272.09) .. controls (393.55,275.48) and (389.64,277.2) .. (385.98,275.94) .. controls (382.33,274.67) and (380.32,270.89) .. (381.5,267.5) -- cycle ;
\draw  [fill={rgb, 255:red, 0; green, 0; blue, 0 }  ,fill opacity=1 ] (162.5,90.5) .. controls (162.5,86.91) and (165.63,84) .. (169.5,84) .. controls (173.37,84) and (176.5,86.91) .. (176.5,90.5) .. controls (176.5,94.09) and (173.37,97) .. (169.5,97) .. controls (165.63,97) and (162.5,94.09) .. (162.5,90.5) -- cycle ;
\draw   (163,71) .. controls (163,67.41) and (166.13,64.5) .. (170,64.5) .. controls (173.87,64.5) and (177,67.41) .. (177,71) .. controls (177,74.59) and (173.87,77.5) .. (170,77.5) .. controls (166.13,77.5) and (163,74.59) .. (163,71) -- cycle ;
\draw (127.5,151.9) node [anchor=north west][inner sep=0.75pt]    {$-l_{1}$};
\draw (265,150.4) node [anchor=north west][inner sep=0.75pt]    {$0$};
\draw (413,44.4) node [anchor=north west][inner sep=0.75pt]    {$l_{2}$};
\draw (459.5,96.4) node [anchor=north west][inner sep=0.75pt]    {$l_{3}$};
\draw (383,162.4) node [anchor=north west][inner sep=0.75pt]    {$l_{4}$};
\draw (396.73,275.49) node [anchor=north west][inner sep=0.75pt]    {$l_{N}$};
\draw (182,62) node [anchor=north west][inner sep=0.75pt]   [align=left] {no control};
\draw (182.5,80) node [anchor=north west][inner sep=0.75pt]   [align=left] {control};
\draw (394.48,189.49) node [anchor=north west][inner sep=0.75pt]  [rotate=-89.95] [align=left] {$\cdots$ $\cdots$};
\end{tikzpicture}
		\caption{A compact graph with $N+1$ edges}
		\label{control}
	\end{figure}
Thus, in this work, we consider the following system
\begin{equation}\label{graph}
\begin{cases}
i\pt u_j +\px^2 u_j - \px^4 u_j =0, & (x,t)\in  I_j \times (0,T),\ j=1,2, ..., N\\
u_j(x,0)=u_{j0}(x),                                   & x\in I_j,\ j=1,2, ..., N\\
\end{cases}
\end{equation}
with appropriated boundary conditions as follows
\begin{equation}\label{bound}
\left\lbrace
\begin{split}
&u_1(-l_1,t )=\partial_x u_1(-l_1,t)=0,\\
& u_j(l_j,t)=0, & j \in \{2,3,\cdots, N\},\\
& \partial_x u_j(l_j,t)= h_j(t), &j \in \{2,3,\cdots, N\}, \\
& u_1(0,t)=\alpha_j u_j(0,t),  & j \in \{2,3,\cdots, N\},\\
&\partial_x u_1(0,t)=\sum_{j=2}^{N}\frac{\partial_x u_j(0,t)}{\alpha_j} \\
& \partial_x^2 u_1(0)= \alpha_j \partial_x^2 u_j(0,t), & j \in \{2,3,\cdots, N\}, \\
&\partial_x^3 u_1(0,t)=\sum_{j=2}^{N}\frac{\partial_x^3 u_j(0,t)}{\alpha_j}.
\end{split}\right.
\end{equation}
Here $u_{j}(x, t)$ is the amplitude of propagation of intense laser beams on the edge $e_{j}$ at position $x \in I_{j}$ at time $t, h_{j}=h_{j}(t)$ is the control on the edge $e_{j}\ (j \in\{2, \cdots, N+1\})$ belonging to $L^{2}(0, T)$ and $\alpha_{j} \ (j \in\{2, \cdots, N+1\})$ is a positive constant. The initial data $u_{j 0}$ are supposed to be $H^{-2}(\mathcal{G})$ functions of the space variable.

With this framework in hand, our work deals with the following classical control problem.

\vspace{0.2cm}

\noindent\textbf{Boundary controllability problem:}\textit{ For any $T > 0$, $l_j > 0$, $u_{j0}\in H^{-2}(\mathcal{G})$ and $u_{T}\in H^{-2}(\mathcal{G})$, is it possible to find $N$ Neumann boundary controls $h_j\in L^2(0,T)$ such that the solution $u$ of \eqref{graph}-\eqref{bound} on the tree shaped network of $N+1$ edges (see Fig. \ref{control}) satisfies 
\begin{equation}\label{ect}
u(\cdot,0) = u_{0}(\cdot) \quad \text{ and } \quad u(\cdot,T)=u_{T}(\cdot) ?
\end{equation}}

The answer for that question is given by the following result.
\begin{theorem}\label{Th_Control_N}
For $T>0$ and $l_1, l_2, \cdots l_{N}$ positive real numbers, let us suppose that 
\begin{equation}\label{condition1}
T >  \sqrt{ \frac{ \overline{L} (L^2+\pi^2)}{\pi^2\varepsilon(1-  \overline{L} \varepsilon)}}:=T_{min}
\end{equation}
where
\begin{equation}\label{condition2}
L=\max \left\lbrace l_1, l_2 \cdots, l_{N}\right\rbrace, \quad  \overline{L} =\max \left\lbrace 2l_1, \max \left\lbrace l_2,l_3, \cdots, l_N   \right\rbrace  + l_1\right\rbrace, 
\end{equation}
and
\begin{equation}\label{condition3}
 0<\varepsilon < \frac{1}{ \overline{L} }.
 \end{equation}
Additionally, suppose that the coefficients of the boundary conditions \eqref{bound} satisfies
\begin{equation}\label{putaquepario1}
\sum_{j=2}^{N}\frac{1}{\alpha^2_j}=1 \quad \text{and} \quad \frac{1}{\alpha^2_j}\leq \frac{1}{N-1}.
\end{equation}
Then for any $u_{0},\ u_{T}\in H^{-2}(\mathcal{G})$, there exists a control $h_j(t)\in L^2(0,T)$, for $j=2,...,N$, such that the unique solution $u(x,t)\in C([0,T];H^{-2}(\mathcal{G}))$ of \eqref{graph}-\eqref{bound}, with $h_1(t)=0$, satisfies \eqref{ect}.
\end{theorem}

\subsection{Outline and structure of the paper} In this article we prove the exact controllability of the Schrödinger type system with mixed dispersion in star graph structure $\mathcal{G}$ of $(N+1)$ edges $e_{j}$ of lengths $l_{j}>0, j \in\{1, . ., N+1\}$, connected at one vertex that we assume to be $0$ for all the edges (see Fig. \ref{control}). Precisely, we are able to prove that solutions of adjoint system associated to \eqref{graph}, with boundary conditions \eqref{bound}, preserve conservation laws in $L^2(\mathcal{G})$, $H^1(\mathcal{G})$ and $H^2(\mathcal{G})$ (see Appendix \ref{Sec4}), which are proved \textit{via} Morawetz multipliers. With this in hand, an \textit{observability inequality} associated with the solution of the adjoint system is proved. Here, the relation between $T>T_{min}$, where  $$T_{min}=\sqrt{ \frac{ \overline{L} (L^2+\pi^2)}{\pi^2\varepsilon(1-  \overline{L} \varepsilon)}},$$ is crucial to prove the result.

\begin{remark}Let us give some remarks in order. 
\begin{itemize}
\item[1.] It is important to point out that the transmission conditions at the central node 0 are inspired by the recent papers \cite{CaCaGa,Cav,GM,MNS}. It is not the only possible choice, and the main motivation is that they guarantee uniqueness of the regular solutions of the $(Bi-NLS)$ equation linearized around 0.
\item[2.] An important fact is that we are able to deal with the mix dispersion in the system \eqref{graph}, that is, with laplacian and bi-laplacian terms in the system. The laplacian term gives us an extra difficulty to deal with the adjoint system associated to \eqref{graph}. Precisely, if we remove the term $\partial^2_x$ in \eqref{graph} and deal only with the  fourth order Schrödinger equation with the boundary conditions \eqref{bound} we can use two different constants $\alpha_j$ and $\beta_j$  in the traces of the boundary conditions.
\item[3.] We are able to control $N+1$ edges with $N-$boundary controls, however, we do not have the sharp conditions on the lengths $l_j$. Moreover, the time of control $T>T_{min}$ is not sharp, but we get an explicit constant in the observability inequality. In this way, these two problems are open.   
\end{itemize}
\end{remark}  

To end our introduction, we present the outline of the manuscript. Section \ref{Sec2} is related with the well-posedness results for the system \eqref{graph}-\eqref{bound} and its adjoint. In Section \ref{Sec3}, we give a rigorous proof of observability inequality, and with this in hand, we are able to prove Theorem \ref{Th_Control_N}. In Appendix \ref{Sec4} we present key lemmas using Morawetz multipliers which are crucial to prove the main result of the paper.


\section{Well-posedness results}\label{Sec2}
We first study the homogeneous linear system (without control) and the adjoint system associated to \eqref{graph}-\eqref{bound}. After that, the linear biharmonic Schrödinger equation with regular initial data and controls is studied. 

\subsection{Study of the linear system} In this section we consider the following linear model
\begin{equation}\label{graph_1}
\begin{cases}
i\pt u_j +\px^2 u_j - \px^4 u_j =0, & (t,x)\in (0,T) \times   I_j,\ j=1,2, ..., N\\
u_j(0,x)=u_{j0}(x),                                   & x \in I_j,\ j=1,2, ..., N\\
\end{cases}
\end{equation}
with the boundary conditions
\begin{equation}\label{bound1}
\left\lbrace
\begin{split}
&u_1(-l_1,t )=\partial_x u_1(-l_1,t)=0\\
&u_j(l_j,t)= \partial_x u_j(l_j,t)=0, & j \in \{2,3,\cdots, N\}, \\
& u_1(0,t)=\alpha_j u_j(0,t),  & j \in \{2,3,\cdots, N\}, \\
&\partial_x u_1(0,t)=\sum_{j=2}^{N}\frac{\partial_x u_j(0,t)}{\alpha_j} \\
& \partial_x^2 u_1(0)= \alpha_j \partial_x^2 u_j(0,t), &j \in \{2,3,\cdots, N\}, \\
&\partial_x^3 u_1(0,t)=\sum_{j=2}^{N}\frac{\partial_x^3 u_j(0,t)}{\alpha_j}.
\end{split}\right.
\end{equation}
Additionally, from now on we use the notation introduced in the introduction of the manuscript. 

Let us consider the differential operator $$A: u=\left(u_{1}, \cdots, u_{N+1}\right) \in \mathcal{D}(A) \subset L^{2}(\mathcal{G}) \mapsto i \partial_x^2 u - i \partial_x^4 u \in L^{2}(\mathcal{R})$$
with domain defined by
\begin{equation}\label{b1}
D(A):= \left\lbrace u \in \prod_{j=1}^N H^4(I_j) \cap V   :  	 \partial_x u_1(0)=\sum_{j=2}^{N}\frac{\partial_x u_j(0)}{\alpha_j},  \,\, \partial_x^3 u_1(0)=\sum_{j=2}^{N}\frac{\partial_x^3 u_j(0)}{\alpha_j}  \right\rbrace,
\end{equation}
where
\begin{multline}\label{b2}
V=\left\lbrace   u \in \prod_{j=1}^N H^2(I_j) :  u_1(-l_1 )=\partial_x u_1(-l_1)=u_j(l_j)=\partial_x u_j(l_j)=0,  \right. \\
\left. u_1(0)=\alpha_j u_j(0),  \partial_x^2 u_1(0)= \alpha_j \partial_x^2 u_j(0), \quad j \in \{2,3,\cdots, N\}. 	\right\rbrace
\end{multline}

Then we can rewrite the homogeneous linear system \eqref{graph_1}-\eqref{bound1} takes the form
\begin{equation}\label{gen}
\begin{cases}
u_t (t)= Au(t),  & t>0 \\
u(0)=u_0 \in L^2(\G).
\end{cases}
\end{equation}
The following proposition guarantees some properties for the operator $A$. Precisely, the following holds.
\begin{proposition}\label{selfadjoint}
The operator $A:D(A)\subset L^2(\mathcal G)\rightarrow L^2(\mathcal G)$ is self-adjoint in $L^2(\G)$.
\end{proposition}
\begin{proof}
 Let us first to prove that $A$ is a symmetric operator. To do this let $u$ and $v$ in $D(A)$. Then, by approximating $u$ and $v$ by $C^4(\mathcal{G})$ functions, integrating by parts and using the boundary conditions \eqref{b1} and \eqref{b2}  we have that
\begin{equation*}
\begin{split}
(Au,v)_{L^{2}(\mathcal{G})}&
= \re\int_{-l_1}^0   (Au)_1(x)\overline{v_1(x)}dx + \re \sum_{j=2}^{N} \int_{0}^{l_j}   (Au)_j(x)\overline{v_j(x)}dx\\
&=\re \int_{-l_1}^0   (i \partial_x^2 u_1 - i \partial_x^4 u_1)\overline{v_1(x)}dx + \re \sum_{j=2}^{N}\int_{0}^{l_j}   (i \partial_x^2 u_j - i \partial_x^4 u_j)(x)\overline{v_j(x)}dx\\
&=\re \int_{-l_1}^0  {u_1} (i \partial_x^2 \overline{v}_1 - i \partial_x^4 \overline{v}_1)dx +\re \sum_{j=2}^{N}\int_{0}^{l_j} {u_j}  (i \partial_x^2 \overline{v}_j - i \partial_x^4 \overline{v}_j)dx\\
&\quad \quad + \re i\sum_{j=1}^{N} \left[\partial_xu_j\overline{v}_j-u_j\partial_x\overline{v}_j-\partial_x^3u_j\overline{v}_j+\partial_x^2u_j\partial_x\overline{v}_j-\partial_xu_j\partial_x^2\overline{v}_j+u_j\partial_x^3\overline{v}_j\right]_{\partial\mathcal{G}}\\
&=(u,Av)_{L^2(\mathcal{G})},\ \forall\ u,v\in D(A),
\end{split}
\end{equation*}
that is, $A$ is symmetric. It is not hard to see that $D(A^*)=D(A)$, so $A$ is self-adjoint. This finishes the proof.
\end{proof}

By using semigroup theory, $A$ generates a strongly continuous unitary group on $L^2(\G)$, and for any $u_0=(u_{10}, u_{20},..., u_{N0}) \in  L^2(\G)$ there exists a unique mild solution $u \in C([0; T];L^2(\G))$ of \eqref{gen}. Furthermore, if $u_0 	\in D(A)$, then \eqref{gen} has a classical solution satisfying  $u \in C([0; T];D(A)) \cap C^1([0; T];L^2(\G))$. Summarizing, we have the following result.  

\begin{proposition}\label{mild}
	Let $u_0=(u_{10}, u_{20},..., u_{N0})\in H_0^k(\mathcal G)$, for $k\in\{0,1,2,3,4\}$. Then the linear system \eqref{graph_1} with boundary conditions \eqref{bound1}  has a unique solution $u$ on the space $C([0,T]:H_0^k(\mathcal G))$. In particular, for $k=4$ we get  a classical solution and for the other cases ($k\in\{0,1,2,3\}$) the solution is a mild solution. 
\end{proposition}

Now, we deal with the adjoint system associated to \eqref{graph_1}-\eqref{bound1}. As the operator $A=i\partial_x^2-i\partial_x^4$ is self adjoint (see Proposition \ref{selfadjoint})  the adjoint system is defined as follows
\begin{equation}\label{adj_graph_1a}
\begin{cases}
i\pt v_j +\px^2 v_j - \px^4 v_j =0, & (t,x)\in (0,T) \times   I_j,\ j=1,2, ..., N\\
v_j(T,x)=v_{jT}(x),                                   & x \in I_j,\ j=1,2, ..., N\\
\end{cases}
\end{equation}
with the boundary conditions
\begin{equation}\label{adj_bound1}
\left\lbrace
\begin{split}
&v_1(-l_1,t )=\partial_x v_1(-l_1,t)=0\\
&v_j(l_j,t)=\partial_x v_j(l_j,t)=0, & j \in \{2,3,\cdots, N\}, \\
& v_1(0,t)=\alpha_j v_j(0,t),  & j \in \{2,3,\cdots, N\}, \\
&\partial_x v_1(0,t)=\sum_{j=2}^{N}\frac{\partial_x v_j(0,t)}{\alpha_j} \\
& \partial_x^2 v_1(0)= \alpha_j \partial_x^2 v_j(0,t), &j \in \{2,3,\cdots, N\}, \\
&\partial_x^3 v_1(0,t)=\sum_{j=2}^{N}\frac{\partial_x^3 v_j(0,t)}{\alpha_j}.
\end{split}\right.
\end{equation}
Also, as $A=A^*$ we have that $D(A^{*})=D(A)$ and the proof of well-posedness is the same that in Proposition \ref{mild}.


\section{Exact boundary controllability}\label{Sec3}
This section is devoted to the analysis of the exact controllability property for the
linear system corresponding to \eqref{graph} with boundary control \eqref{bound}. Here, we will present the answer for the control problem presented in the introduction of this work. First, let us present two definitions that  will be important for the rest of the work.
\begin{definition}
Let $T > 0$, The system \eqref{graph}-\eqref{bound} is exactly controllable in time $T$ if for any initial and final data $u_0\,, u_T\in H^{-2}(\mathcal{G})$  there exist control functions $h_j\in L^2(0,T)$, $j\in \{2,3,\cdots, N\}$, such that solution $u$ of \eqref{graph}-\eqref{bound} on the tree shaped network of $N+1$ edges satisfies \eqref{ect}. In addition, when $u_T=0$ we said that the system \eqref{graph}-\eqref{bound} is null controllable in time $T$.
\end{definition}

Now on, consider the transposition solution to \eqref{graph}-\eqref{bound}, with $h_1(t)=0$, which is given by the following.
\begin{definition}
We say $u\in L^{\infty}(0,T;H^{-2}(\mathcal{G}))$ is solution of \eqref{graph}-\eqref{bound}, with $h_1(t)=0$,
in the transposition sense if and only if 
$$
\sum_{j=1}^{N}\left(\int_0^T\left\langle u_j(t),f_j(t) \right\rangle dt+i\langle u_j(0),v_j(0)\rangle\right)+\sum_{j=2}^{N}\left(\int_{\mathcal{G}}h_j\partial^2_x\overline{v_j}dx\right)=0,
$$
for every $f\in L^2(0,T;H^2_0(\mathcal{G}))$, where $v(x,t)$ is the mild solution to the problem \eqref{adj_graph_1a}-\eqref{adj_bound1} on the space $C([0,T];H_0^2(G))$, with $v(x,T)=0$, obtained in Proposition \ref{mild}. Here, $\left\langle \cdot,\cdot \right\rangle$ means the duality between the spaces  $H^{-2}(\mathcal{G})$ and $H^2_0(\mathcal{G})$.
 \end{definition}
 
With this in hand, the following lemma gives an equivalent condition for the exact controllability property.

\begin{lemma}\label{L_CECP_1}
Let $u_T\in H^{-2}(\mathcal{G})$. Then, there exist controls  $h_j(t)\in L^2(0,T)$, for $j=2,...,N$, such that the solution $u(x,t)$ of \eqref{graph}-\eqref{bound}, with $h_1(t)=0$, satisfies \eqref{ect} if and only if
\begin{equation}\label{CECP_1}
i\sum_{j=1}^N\int_{I_j}\langle u_j(T),\overline{v}_j(T)\rangle dx=\sum_{j=2}^{N}\int_0^Th_j(t)\partial^2_xv_j(l_j,t)dt,
\end{equation}
where $v$ is solution of \eqref{adj_graph_1a}-\eqref{adj_bound1}, with initial data $v(x,T)=v(T)$. 
\end{lemma}
\begin{proof}
Relation \eqref{CECP_1} is obtained multiplying \eqref{graph}-\eqref{bound}, with $h_1(t)=0$, by the solution $v$ of \eqref{adj_graph_1a}-\eqref{adj_bound1}  and integrating by parts on $\mathcal{G}\times(0,T)$.
\end{proof}

\subsection{Observability inequality} A fundamental role will be played by the following observability result, which together with Lemma \ref{L_CECP_1}  give us Theorem \ref{Th_Control_N}.

\begin{proposition}\label{Oinequality_1}
Let $l_j>$ for any $j\in \{1,\cdots,N+1\}$ satisfying \eqref{condition1} and assume
that \eqref{putaquepario1} holds. There exists a positive constant $T_{min}$ such that if $T>T_{min}$, then the following inequality holds
\begin{equation}\label{OI_2}
\norm{v(x,T)}_{H^2_0(\mathcal{G})}^2\leq C\sum_{j=2}^{N}\norm{\partial^2_xv_j(l_j,t)}^2_{L^2(0,T)}
\end{equation}
for any $v=\left(v_{1}, v_{2}, \cdots, v_{N+1}\right)$ solution of \eqref{adj_graph_1a}-\eqref{adj_bound1} with final condition $v_{T}=\left(v_{1}^{T}, v_{2}^{T}, \cdots, v_{N+1}^{T}\right) \in H^2_0(\mathcal{G})$ and for a positive constant $C >0$.
\end{proposition}

\begin{proof}
Firstly, taking $f=0$ and choosing $q(x,t)=1$ in \eqref{identity}, we get that
\begin{multline*}
\begin{split}
-\frac{Im}{2}\left.\int_{\mathcal{G}}v\overline{\partial_xv}\right]_0^Tdx+\frac{Im}{2}\left.\int_0^Tv\overline{\partial_tv}\right]_{\partial\mathcal{G}}dt&+\frac{1}{2}\left.\int_0^T|\partial_xv|^2\right]_{\partial\mathcal{G}}dt +\frac{1}{2}\left.\int_0^T|\partial_x^2v|^2\right]_{\partial\mathcal{G}}dt \\&-Re\int_0^T\left[ \partial_x^3 v\overline{\partial_xv}\right]_{\partial\mathcal{G}}dt=0,
\end{split}
\end{multline*}
or equivalently,
\begin{multline*}
\begin{split}
0=&-\frac{Im}{2}\left.\int_{\mathcal{G}}v\overline{\partial_xv}\right]_0^Tdx+\frac{Im}{2}\int_0^T \left( \left[ v_1\overline{\partial_t v}_1\right]_{-l_1}^0 + \left[ \sum_{j=2}^{N} v_j\overline{\partial_t v}_j\right]_{0}^{l_j}\right) dt  \\
& +\frac{1}{2}\int_0^T\left( \left[ |\partial_x v_1|^2 \right]_{-l_1}^0 + \left[ \sum_{j=2}^{N}  |\partial_x v_j|^2 \right]_{0}^{l_j} \right)dt  
+\frac{1}{2}\int_0^T\left( \left[ |\partial_x^2 v_1|^2 \right]_{-l_1}^0 + \left[ \sum_{j=2}^{N}  |\partial_x^2 v_j|^2 \right]_{0}^{l_j} \right)dt\\
& -Re\int_0^T\left(\left[ \partial_x^3 v_1\overline{\partial_x v}_1\right]_{-l_1}^0 + \left[ \sum_{j=2}^{N} \partial_x^3 v_j\overline{ \partial_x v}_j\right]_{0}^{l_j}\right)dt.
 \end{split}
\end{multline*}
By using the boundary conditions \eqref{adj_bound1}, it follows that
\begin{multline*}
\begin{split}
0=&-\frac{Im}{2}\left.\int_{\mathcal{G}}v\overline{\partial_xv}\right]_0^Tdx+\frac{Im}{2}\int_0^T \left(  v_1(0)\overline{\partial_t v}_1(0)  - \sum_{j=2}^{N} v_j(0)\overline{\partial_t v}_j(0)\right) dt \\
&+\frac{1}{2}\int_0^T\left(  |\partial_x v_1(0)|^2- \sum_{j=2}^{N}  |\partial_x v_j(0)|^2 \right)dt  \\
 &+\frac{1}{2}\int_0^T\left(  |\partial_x^2 v_1(0)|^2  -   |\partial_x^2 v_1(-l_1)|^2   +  \sum_{j=2}^{N} \left(  |\partial_x^2 v_j(l_j)|^2 - |\partial_x^2 v_j(0)|^2 \right) \right)dt \\
&-Re\int_0^T\left(\partial_x^3 v_1 (0)\overline{\partial_x v}_1(0)  -  \sum_{j=2}^{N} \partial_x^3 v_j(0)\overline{ \partial_x v}_j(0) \right)dt.
\end{split}
\end{multline*}
Once again, due to the boundary conditions \eqref{adj_bound1} and relations \eqref{putaquepario1}, we have that 
\begin{equation}\label{perfect}
\begin{split}
\int_0^T   |\partial_x^2&v_1(-l_1)|^2 dt =-Im\left.\int_{\mathcal{G}}v\overline{\partial_xv}\right]_0^Tdx  +\int_0^T    \sum_{j=2}^{N}  |\partial_x^2 v_j(l_j)|^2  dt \\
 &+\int_0^T\left(  |\partial_x v_1(0)|^2- \sum_{j=2}^{N}  |\partial_x v_j(0)|^2 \right)dt 
+\int_0^T\left(  |\partial_x^2 v_1(0)|^2- \sum_{j=2}^{N}  |\partial_x^2 v_j(0)|^2 \right)dt.
\end{split}
\end{equation}
 Thanks to relations \eqref{putaquepario1}, we deduce that 
\begin{multline*}
\int_0^T\left(  |\partial_x v_1(0)|^2- \sum_{j=2}^{N}  |\partial_x v_j(0)|^2 \right)dt = \frac{1}{2}\int_0^T\left( \left| \sum_{j=2}^{N} \frac{\partial_x v_j(0)}{\alpha_j}\right |^2- \sum_{j=2}^{N}  |\partial_x v_j(0)|^2 \right)dt \\
\leq \frac{1}{2}\int_0^T\left( (N-1) \sum_{j=2}^{N} \left|\frac{\partial_x v_j(0)}{\alpha_j}\right |^2- \sum_{j=2}^{N}  |\partial_x v_j(0)|^2 \right)dt = \frac{1}{2}\int_0^T \sum_{j=2}^{N}  |\partial_x v_j(0)|^2\left( \frac{N-1}{\alpha_j^2}-1 \right)dt \leq 0 
\end{multline*} 
and
\begin{align*}
\frac{1}{2}\int_0^T\left(  |\partial_x^2 v_1(0)|^2- \sum_{j=2}^{N}  |\partial_x^2 v_j(0)|^2 \right)dt &= \frac{1}{2}\int_0^T\left( |\partial_x^2 v_1(0)|^2- \sum_{j=2}^{N}  \left| \frac{\partial_x^2 v_1(0)}{\alpha_j} \right|^2 \right)dt \\
&= \frac{1}{2}\int_0^T |\partial_x^2 v_1(0)|^2\left( 1- \sum_{j=2}^{N}  \frac{1}{\alpha_j^2} \right)dt= 0.
\end{align*} 
Thus, previous calculations ensure that
\begin{equation}\label{newmult_1}
\int_0^T   |\partial_x^2 v_1(-l_1)|^2 dt \leq Im\left.\int_{\mathcal{G}}v\overline{\partial_xv}\right]_0^Tdx  +\int_0^T    \sum_{j=2}^{N}  |\partial_x^2 v_j(l_j)|^2  dt .
\end{equation}

Now,   choosing  $q(x,t)=x$ in \eqref{identity}, by using the boundary conditions \eqref{adj_bound1} and taking $f=0$, we get
\begin{align*}
2\int_Q|\partial_xv|^2dxdt+4\int_Q|\partial_x^2v|^2dxdt &=\left.\int_0^T|\partial^2_xv|^2x\right]_{\partial\mathcal{G}}dt-Im\left.\int_{\mathcal{G}}v\overline{\partial^2_xv}\right]_0^Tdx \\
&=l_1 \int_0^T   |\partial_x^2 v_1(-l_1)|^2 dt      + \int_0^T  \sum_{j=2}^{N} |\partial_x^2 v_j(l_j)|^{2}l_jdt-Im\left.\int_{\mathcal{G}}v\overline{\partial_xv}x\right]_0^Tdx.
\end{align*}
From inequality \eqref{newmult_1}, it yields that
\begin{multline*}
\begin{split}
2\int_Q|\partial_xv|^2dxdt+4\int_Q|\partial_x^2v|^2dxdt \leq& \ l_1Im\left.\int_{\mathcal{G}}v\overline{\partial_xv}\right]_0^Tdx  +l_1\int_0^T    \sum_{j=2}^{N}  |\partial_x^2 v_j(l_j)|^2  dt     \\
  & + \int_0^T  \sum_{j=2}^{N} |\partial_x^2 v_j(l_j)|^{2}l_jdt-Im\left.\int_{\mathcal{G}}v\overline{\partial_xv}x\right]_0^Tdx,
   \end{split}
\end{multline*}
hence
\begin{equation}\label{OI1_b_1}
2\int_Q|\partial_xv|^2dxdt+4\int_Q|\partial_x^2v|^2dxdt\leq Im\left.\int_{\mathcal{G}}v\overline{\partial_xv} (l_1-x)\right]_0^Tdx  +(\overline{L}+l_1)\int_0^T    \sum_{j=2}^{N}  |\partial_x^2 v_j(l_j)|^2  dt     ,
\end{equation}
where $\overline{L}$ is defined by \eqref{condition2}.

Now, we are in position to prove \eqref{OI_2}. Thanks to \eqref{OI1_b_1},  follows that
\begin{multline}\label{OI1_c_1}
\begin{split}
2\int_Q(|\partial_xv|^2+|\partial_x^2v|^2)dxdt\leq&\ (\overline{L}+l_1)\int_0^T    \sum_{j=2}^{N}  |\partial_x^2 v_j(l_j)|^2  dt  + \left|\left.\int_{\mathcal{G}}v\overline{\partial_xv} (l_1-x)\right]_0^Tdx \right|     \\
\leq & \ L\int_0^T    \sum_{j=2}^{N}  |\partial_x^2 v_j(l_j)|^2  dt + \int_{\mathcal{G}}|v(T)||\overline{\partial_xv(T)}| |l_1-x|dx  \\
&+ \int_{\mathcal{G}}|v(0)||\overline{\partial_xv(0)}| |l_1-x|dx .
\end{split}
\end{multline}
As we have the conservation laws for solutions of \eqref{adj_bound1aa}, that is, \eqref{H12_c} is satisfied, so by using it on the left hand side of \eqref{OI1_c_1}, yields that
\begin{multline}\label{OI1_d_1}
\begin{split}
2\int_0^T(\norm{\partial_xv(T)}^2_{L^2(\mathcal{G})}+&\norm{\partial^2_xv(T)}^2_{L^2(\mathcal{G})})dt
\leq  \ L\int_0^T    \sum_{j=2}^{N}  |\partial_x^2 v_j(l_j)|^2  dt \\ & + \overline{L} \left( \int_{\mathcal{G}}|v(T)||\overline{\partial_xv(T)}| dx   + \int_{\mathcal{G}}|v(0)||\overline{\partial_xv(0)}| |dx\right). 
\end{split}
\end{multline}
Applying Young inequality in \eqref{OI1_d_1}, with  $\varepsilon >0$ satisfying \eqref{condition3}, we deduce that
\begin{multline}\label{OI1_d_1_2}
2T \left(\norm{\partial_xv(T)}^2_{L^2(\mathcal{G})}+\norm{\partial_x^2v(T)}^2_{L^2(\mathcal{G})}\right)
\leq \ L \int_0^T    \sum_{j=2}^{N}  |\partial_x^2 v_j(l_j)|^2  dt \\  + \overline{L}  \left( \frac{1}{\varepsilon T} \int_{\mathcal{G}}|v(T)|^2  dx + \varepsilon T \int_{\mathcal{G}}  |\overline{\partial_xv(T)}|^2 dx   + \frac{1}{\varepsilon T} \int_{\mathcal{G}}|v(0)|^2 dx + \varepsilon  T \int_{\mathcal{G}}|\overline{\partial_xv(0)})|^2 dx\right). 
\end{multline}
Therefore, we have due to \eqref{OI1_d_1_2} and, again using the conservation law, the following estimate
\begin{equation*}
2T(1-\overline{L}\epsilon ) \left( \norm{\partial_xv(T)}^2_{L^2(\mathcal{G})}+\norm{\partial_x^2v(T)}^2_{L^2(\mathcal{G})}\right)
  \leq   L \int_0^T    \sum_{j=2}^{N}  |\partial_x^2 v_j(l_j)|^2  dt +\frac{2\overline{L}  }{\varepsilon T} \norm{v(T)}^2_{L^2(\mathcal{G})}.
\end{equation*}
From relation \eqref{poincare}, we have that 
\begin{multline*}
\begin{split}
2(1-\overline{L}\epsilon )T \left(\norm{\partial_xv(T)}^2_{L^2(\mathcal{G})}+\norm{\partial_x^2v(T)}^2_{L^2(\mathcal{G})}\right)\leq&\ L\int_0^T    \sum_{j=2}^{N}  |\partial_x^2 v_j(l_j)|^2  dt\\ & +\frac{2 M }{\varepsilon T}   \left( \frac{L^2}{\pi^2} +1 \right) (\norm{\partial_xv(T)}^2_{L^2(\mathcal{G})}+\norm{\partial^2_xv(T)}^2_{L^2(\mathcal{G})}).
\end{split}
\end{multline*}
Equivalently, we get that
\begin{equation*}
2\overline{L}\left[ \left( \frac{1}{ \overline{L}} -\epsilon \right) T    - \frac{1}{\varepsilon T}   \left( \frac{L^2}{\pi^2} +1 \right)\right]   \left(\norm{\partial_xv(T)}^2_{L^2(\mathcal{G})}+\norm{\partial_x^2v(T)}^2_{L^2(\mathcal{G})} \right)
\leq L\int_0^T    \sum_{j=2}^{N}  |\partial_x^2 v_j(l_j)|^2  dt.
\end{equation*}
 Note that the conditions \eqref{condition1}, \eqref{condition2}  and \eqref{condition3} imply that 
\begin{equation*}
K=\left[ \left( \frac{1}{\overline{L}} -\epsilon \right) T    - \frac{1}{\varepsilon T}   \left( \frac{L^2}{\pi^2} +1 \right)\right] >0
\end{equation*}
Thus, again using \eqref{poincare}, we achieved the observability inequality \eqref{OI_2}.
\end{proof}

\subsection{Proof of Theorem  \ref{Th_Control_N}}
Notice that Theorem \ref{Th_Control_N} is a consequence  of the observability inequality \eqref{OI_2}. In fact, without loss of generality, pick $u_0=0$ on $\mathcal{G}$. Define  $\Gamma$ the linear and bounded map
$\Gamma:H^2_0(\mathcal{G})\longrightarrow H^{-2}(\mathcal{G})$
by
\begin{equation}
\Gamma(v(\cdot,T))=\langle u(\cdot,T),\overline{v}(\cdot,T)\rangle,
\end{equation}
where $v=v(x,t)$ is solution of \eqref{adj_graph_1a}-\eqref{adj_bound1}, with initial data $v(x,T)=v(T)$,$\left\langle \cdot,\cdot \right\rangle$ means the duality between the spaces  $H^{-2}(\mathcal{G})$ and $H^2_0(\mathcal{G})$,  $u=u(x,t)$ is solution of \eqref{graph}-\eqref{bound}, with  $h_1(t)=0$ and 
\begin{equation}\label{control2a}
h_j(t)=\partial^2_xv_j(l_j,t),
\end{equation}
for $j=2,\cdots,N$. 

According to Lemma \ref{L_CECP_1} and Proposition \ref{Oinequality_1}, we obtain
\begin{equation*}
\langle \Gamma(v(T)),v(T)\rangle =\sum_{j=2}^{N}\norm{h_j(t)}^2_{L^2(0,T)}\geq C^{-1}\norm{v(T)}^2_{H^2_0(\mathcal{G})}.
\end{equation*}
Thus, by the Lax–Milgram theorem, $\Gamma$ is invertible. Consequently, for given  $u(T)\in H^{-2}(\mathcal{G})$, we can define $v(T):=\Gamma^{-1}(u(T))$ which one  solves \eqref{adj_graph_1a}-\eqref{adj_bound1}. Then, if $h_j(t)$, for $j=2,\cdots,N$ is defined by \eqref{control2a}, the corresponding solution $u$ of the system \eqref{graph}-\eqref{bound}, satisfies \eqref{ect} and so, Theorem \ref{Th_Control_N} holds.

\appendix 

\section{Auxiliary lemmas}\label{Sec4}

\subsection{Morawetz multipliers}
This section is dedicated to establishing fundamental identities by the multipliers method, which will be presented in two lemmas. For $f\in L^2(0,T;H^2_0(\mathcal{G}))$, let us consider the following system
\begin{equation}\label{adj_graph_1aa}
\begin{cases}
i\pt u_j +\px^2 u_j - \px^4 u_j =f, & (t,x)\in (0,T) \times   I_j,\ j=1,2, ..., N\\
u_j(0,x)=u_{j0}(x),                                   & x\in I_j,\ j=1,2, ..., N\\
\end{cases}
\end{equation}
with the boundary conditions
\begin{equation}\label{adj_bound1aa}
\left\lbrace
\begin{split}
&u_1(-l_1,t )=\partial_x u_1(-l_1,t)=0\\
& u_j(l_j,t)= \partial_x u_j(l_j,t)=0, \quad j \in \{2,3,\cdots, N\}, \\
& u_1(0,t)=\alpha_j u_j(0,t),  \quad j \in \{2,3,\cdots, N\}, \\
&\partial_x u_1(0,t)=\sum_{j=2}^{N}\frac{\partial_x u_j(0,t)}{\alpha_j} \\
& \partial_x^2 u_1(0)= \alpha_j \partial_x^2 u_j(0,t), \quad j \in \{2,3,\cdots, N\}, \\
&\partial_x^3 u_1(0,t)=\sum_{j=2}^{N}\frac{\partial_x^3 u_j(0,t)}{\alpha_j}.
\end{split}\right.
\end{equation}
The first lemma gives us an identity which will help us to prove the main result of this article. 
\begin{lemma}\label{Id_Obs}
Let $q=q(x,t)\in C^4(\overline{\mathcal{G}}\times(0,T),\mathbb{R})$ with $\overline{\mathcal{G}}$ being the closed set of $\mathcal{G}$. For every solution of \eqref{adj_bound1aa}-\eqref{adj_graph_1aa} with $f \in\mathcal{D}(\mathcal{G})$ and $u_0\in\mathcal{D}(\mathcal{G})$, the following identity holds:
\begin{equation}\label{identity}
\begin{split}
&\frac{Im}{2}\int_Qu\overline{\partial_xu}\partial_tqdxdt-\int_Q|\partial_xu|^2\partial_xqdxdt-2\int_Q|\partial_x^2u|^2\partial_xqdxdt-\frac{Re}{2}\int_{Q}\partial_xu\overline{u}\partial_x^2qdxdt\\
&+\frac{3}{2}\int_{Q}|\partial_xu|^2\partial_x^3qdxdt+\frac{Re}{2}\int_Q\partial_xu\overline{u}\partial_x^4qdxdt-\frac{Im}{2}\left.\int_{\mathcal{G}}u\overline{\partial_xu}q\right]_0^Tdx+\frac{1}{2}\left.\int_0^T|\partial_x^2u|^2q\right]_{\partial\mathcal{G}}dt\\
&+\frac{Im}{2}\left.\int_0^Tu\overline{\partial_tu}q\right]_{\partial\mathcal{G}}dt+\frac{1}{2}\left.\int_0^T|\partial_xu|^2q\right]_{\partial\mathcal{G}}dt+\frac{Re}{2}\left.\int_0^T\partial_xu\overline{u}\partial_xq\right]_{\partial\mathcal{G}} dt\\
&+\int_0^T\left[-|\partial_xu|^2\partial_x^2q+\frac{3}{2}Re(\partial_x^2u\overline{\partial_xu})\partial_xq-Re(\partial_x^3u\overline{\partial_xu})q\right]_{\partial\mathcal{G}}dt\\
&+\int_0^T\left[-\frac{Re}{2}(\partial_xu\overline{u})\partial_x^3q+\frac{Re}{2}(\partial_x^2u\overline{u})\partial_x^2q-\frac{Re}{2}(\partial_x^3u\overline{u})\partial_xq\right]_{\partial\mathcal{G}}dt=\int_Qf(\overline{\partial_xu}q+\frac{1}{2}\overline{u}\partial_xq)dxdt,
\end{split}
\end{equation}
where $Q:=\mathcal{G}\times[0,T]$ and $\partial\mathcal{G}$ is the boundary of $\mathcal{G}$.
\end{lemma}
\begin{proof} Multiplying \eqref{adj_graph_1aa} by $\overline{\partial_xu}q+\frac{1}{2}\overline{u}\partial_xq$, we have that
\begin{equation*}
\begin{split}
&\int_Qi\partial_tu(\overline{\partial_xu}q+\frac{1}{2}\overline{u}\partial_xq)dxdt+\int_Q\partial_x^2u(\overline{\partial_xu}q+\frac{1}{2}\overline{u}\partial_xq)dxdt-\int_Q\partial_x^4u(\overline{\partial_xu}q+\frac{1}{2}\overline{u}\partial_xq)dxdt\\
&-\int_Qf(\overline{\partial_xu}q+\frac{1}{2}\overline{u}\partial_xq)dxdt:=I_1+I_2-I_3-\int_Qf(\overline{\partial_xu}q+\frac{1}{2}\overline{u}\partial_xq)dxdt.
\end{split}
\end{equation*}

Now, we split the proof in three steps.

\vspace{0.1cm}

\noindent\textbf{Step 1.} Analysis of $I_1$.

\vspace{0.1cm}

Integrating by parts, several times, on $Q$, we get

\begin{equation*}
\begin{split}
I_1=&-i\int_Qu\overline{\partial_t\partial_xu}qdxdt-i\int_Qu\overline{\partial_xu}\partial_tqdxdt+i\left.\int_{\mathcal{G}}u\overline{\partial_xu}q\right]_0^Tdx+\frac{i}{2}\int_Q\partial_xu\overline{\partial_tu}qdxdt+\frac{i}{2}\int_Qu\overline{\partial_t\partial_xu}qdxdt\\
&-\frac{i}{2}\left.\int_0^Tu\overline{\partial_tu}q\right]_{\partial\mathcal{G}}dt+\frac{i}{2}\int_Q\partial_xu\overline{u}\partial_tqdxdt+\frac{i}{2}\int_Qu\overline{\partial_xu}\partial_tqdxdt-\frac{i}{2}\left.\int_0^Tu\overline{u}\partial_tq\right]_{\partial\mathcal{G}}dt+\frac{i}{2}\left.\int_{\mathcal{G}}u\overline{u}\partial_xq\right]_0^Tdx\\
=&-\frac{i}{2}\int_Qu\overline{\partial_t\partial_xu}qdxdt-\frac{i}{2}\int_Qu\overline{\partial_xu}\partial_tqdxdt-\frac{i}{2}\int_{Q}\partial_t\partial_xu\overline{u}qdxdt+i\left.\int_{\mathcal{G}}u\overline{\partial_xu}q\right]_0^Tdx-\frac{i}{2}\left.\int_0^Tu\overline{\partial_tu}q\right]_{\partial\mathcal{G}}dt\\
&-\frac{i}{2}\left.\int_{\mathcal{G}}u\overline{\partial_xu}q\right]_0^Tdx-\left.\frac{i}{2}\int_{\mathcal{G}}u\overline{u}\partial_xq\right]_0^Tdx+\frac{i}{2}\left.\left[u\overline{u}q\right]_0^T\right]_{\partial\mathcal{G}}-\frac{i}{2}\left.\int_0^T u\overline{u}\partial_tq\right]_{\partial\mathcal{G}}dt+\frac{i}{2}\left.\int_{\mathcal{G}}u\overline{u}\partial_xq\right]_0^Tdx.
\end{split}
\end{equation*}
Thus, putting together the similar terms in the  last equality, yields that
\begin{equation}\label{i1a}
\begin{split}
I_1=&-\frac{i}{2}\int_Q(u\overline{\partial_t\partial_xu}+\partial_t\partial_xu\overline{u})qdxdt-\frac{i}{2}\int_Qu\overline{\partial_xu}\partial_tqdxdt+\frac{i}{2}\left.\int_{\mathcal{G}}u\overline{\partial_xu}q\right]_0^Tdx\\
&-\frac{i}{2}\left.\int_0^Tu\overline{\partial_tu}q\right]_{\partial\mathcal{G}}dt+\frac{i}{2}\left.\left[u\overline{u}q\right]_0^T\right]_{\partial\mathcal{G}}-\frac{i}{2}\left.\int_0^T u\overline{u}\partial_tq\right]_{\partial\mathcal{G}}dt.
\end{split}
\end{equation}
Finally, taking the real part of \eqref{i1a}, we have
\begin{equation}\label{i1}
\begin{split}
Re(I_1)=\frac{Im}{2}\int_Qu\overline{\partial_xu}\partial_tqdxdt-\frac{Im}{2}\left.\int_{\mathcal{G}}u\overline{\partial_xu}q\right]_0^Tdx
+Re\left(-\frac{i}{2}\left.\int_0^Tu\overline{\partial_tu}q\right]_{\partial\mathcal{G}}dt\right).
\end{split}
\end{equation}

\noindent\textbf{Step 2.} Analysis of the Laplacian integral $I_2$.

\vspace{0.1cm}

Integrating by parts, several times, on $Q$ and taking the real part of $I_2$, follows that
\begin{equation*}
\begin{split}
Re(I_2)=&\int_QRe(\partial_x^2u\overline{\partial_xu}q)dxdt+\frac{Re}{2}\left(\int_Q\partial^2_xu\overline{u}\partial_xqdxdt\right)\\
=&-\frac{1}{2}\int_Q|\partial_xu|^2\partial_xqdxdt+\left.\frac{1}{2}\int_0^T|\partial_xu|^2q\right]_{\partial\mathcal{G}}dt-\frac{1}{2}\int_Q|\partial_xu|^2\partial_xqdxdt\\
&-\frac{Re}{2}\left(\int_Q\partial_xu\overline{u}\partial^2_xqdxdt\right)+\frac{Re}{2}\left.\left(\int_0^T\partial_xu\overline{u}\partial_xq\right]_{\partial\mathcal{G}}dt\right).
\end{split}
\end{equation*}
Consequently,
\begin{equation}\label{i2}
\begin{split}
Re(I_2)=&-\int_Q|\partial_xu|^2\partial_xqdxdt-\frac{Re}{2}\left(\int_Q\partial_xu\overline{u}\partial^2_xqdxdt\right)\\
&+\left.\frac{1}{2}\int_0^T|\partial_xu|^2q\right]_{\partial\mathcal{G}}dt+\frac{Re}{2}\left.\left(\int_0^T\partial_xu\overline{u}\partial_xq\right]_{\partial\mathcal{G}}dt\right).
\end{split}
\end{equation}

\noindent\textbf{Step 3.} Analysis of the bi-Laplacian integral $I_3$.

\vspace{0.1cm}

Integrating by parts on $Q$ give us
\begin{equation}\label{i3a}
\begin{split}
-I_3=&\int_Q\partial^3_xu\overline{\partial^2_xu}qdxdt+\int_Q\partial^3_xu\overline{\partial_xu}\partial_xqdxdt-\left.\int_0^T\partial^3_xu\overline{\partial_xu}q\right]_{\partial\mathcal{G}}dt\\
&+\frac{1}{2}\int_Q\partial^3_xu\overline{\partial_xu}\partial_xqdxdt+\frac{1}{2}\int_Q\partial^3_xu\overline{u}\partial^2_xqdxdt-\left.\frac{1}{2}\int_0^T\partial^3_xu\overline{u}\partial_xq\right]_{\partial\mathcal{G}}dt.
\end{split}
\end{equation}
Taking the real part of \eqref{i3a} and integrating, again, by parts on $Q$, we have
\begin{equation*}
\begin{split}
-Re(I_3)=&-\frac{1}{2}\int_Q|\partial^2_xu|^2\partial_xqdxdt+\frac{1}{2}\left.\int_0^T|\partial^2_xu|^2q\right]_{\partial\mathcal{G}}dt-\frac{3}{2}\int_Q|\partial_x^2u|^2\partial_xqdxdt+\int_Q|\partial_xu|^2\partial^3_xqdxdt\\
&-\left.\int_0^T|\partial_xu|^2\partial^2_xq\right]_{\partial\mathcal{G}}dt+\frac{3Re}{2}\left(\left.\int_0^T\partial^2_xu\overline{\partial_xu}\partial_xq\right]_{\partial\mathcal{G}}dt\right)-Re\left(\left.\int_0^T\partial^3_xu\overline{\partial_xu}q\right]_{\partial\mathcal{G}}dt\right)\\
&+\frac{Re}{2}\left(\int_Q|\partial_xu|^2\partial^3_xqdxdt\right)+\frac{Re}{2}\left(\int_Q\partial_xu\overline{u}\partial_x^4qdxdt\right)-\frac{Re}{2}\left.\left(\int_0^T\partial_xu\overline{u}\partial_x^3q\right]_{\partial\mathcal{G}}dt\right)\\
&+\frac{Re}{2}\left.\left(\int_0^T\partial^2_xu\overline{u}\partial_x^2q\right]_{\partial\mathcal{G}}dt\right)-\frac{Re}{2}\left.\left(\int_0^T\partial^3_xu\overline{u}\partial_xq\right]_{\partial\mathcal{G}} dt\right).
\end{split}
\end{equation*}
The last inequality ensure that
\begin{equation}\label{i3}
\begin{split}
-Re(I_3)=&-2\int_Q|\partial_x^2u|^2\partial_xqdxdt+\frac{3}{2}\int_Q|\partial_xu|^2\partial^3_xqdxdt+\frac{Re}{2}\left(\int_Q\partial_xu\overline{u}\partial^4_xqdxdt\right)\\
&+\frac{1}{2}\left.\int_0^T|\partial^2_xu|^2q\right]_{\partial\mathcal{G}}dt-\left.\int_0^T|\partial_xu|^2\partial^2_xq\right]_{\partial\mathcal{G}}dt+\frac{3Re}{2}\left(\left.\int_0^T\partial^2_xu\overline{\partial_xu}\partial_xq\right]_{\partial\mathcal{G}}dt\right)\\
&-Re\left(\left.\int_0^T\partial^3_xu\overline{\partial_xu}q\right]_{\partial\mathcal{G}}dt\right)-\frac{Re}{2}\left.\left(\int_0^T\partial_xu\overline{u}\partial_x^3q\right]_{\partial\mathcal{G}}dt\right)\\
&+\frac{Re}{2}\left.\left(\int_0^T\partial^2_xu\overline{u}\partial^2_xq\right]_{\partial\mathcal{G}}dt\right)-\frac{Re}{2}\left.\left(\int_0^T\partial^3_xu\overline{u}\partial_xq\right]_{\partial\mathcal{G}} dt\right).
\end{split}
\end{equation}

Finally, taking into account \eqref{i1}, \eqref{i2} and \eqref{i3}, we have
$$Re(I_1)+Re(I_2)-Re(I_3)=\int_Qf(\overline{\partial_xu}q+\frac{1}{2}\overline{u}\partial_xq)dxdt,$$
then \eqref{identity} holds.
\end{proof}

\subsection{Conservation laws}  
The final result of the appendix gives us the conservation laws for the solutions of \eqref{adj_graph_1aa}-\eqref{adj_bound1aa}, with $f=0$. More precisely, the result is the following one.

\begin{lemma}\label{conservations}
For any positive time t, the solution $u$ of \eqref{adj_graph_1aa}-\eqref{adj_bound1aa}, with $f=0$, satisfies
\begin{equation}\label{L2_c}
\norm{u(t)}_{L^2(\mathcal{G})}=\norm{u(0)}_{L^2(\mathcal{G})}
\end{equation}
and 
\begin{equation}\label{H12_c}
\norm{\partial_xu(t)}_{L^2(\mathcal{G})}+\norm{\partial_x^2u(t)}_{L^2(\mathcal{G})}=\norm{\partial_xu(0)}_{L^2(\mathcal{G})}+\norm{\partial^2_xu(0)}_{L^2(\mathcal{G})}.
\end{equation}
Additionally, we have
\begin{equation}\label{H2_c}
\norm{u(t)}_{L^2(\mathcal{G})}+\norm{\partial_xu(t)}_{L^2(\mathcal{G})}+\norm{\partial_x^2u(t)}_{L^2(\mathcal{G})}=\norm{u(0)}_{L^2(\mathcal{G})}+\norm{\partial_xu(0)}_{L^2(\mathcal{G})}+\norm{\partial^2_xu(0)}_{L^2(\mathcal{G})}.
\end{equation}
\end{lemma}
\begin{proof} Multiplying system \eqref{adj_graph_1aa}-\eqref{adj_bound1aa}, with $f=0$, by $i\overline{u}$ and integrating in $\mathcal{G}$, we get 
\begin{equation}\label{L2_caa}
0=-\int_{\mathcal{G}}\pt u_j \overline{u}dx+i\int_{\mathcal{G}}\px^2 u_j \overline{u}dx -i\int_{\mathcal{G}} \px^4 u_j\overline{u}dx=L_1+L_2+L_3.
\end{equation}
We are now looking for the integral $L_2+L_3$. Let us, first, rewrite these quantities as follows
\begin{equation}\label{L2_ca}
\begin{split}
L_2+L_3=i\sum_{j=1}^{N}\left(\int_{I_j}\partial^2_xu_j\overline{u}_jdx-\int_{I_j}\partial^4_xu_j\overline{u}_jdx\right).
\end{split}
\end{equation}
Integrating \eqref{L2_ca} by parts and taking the real part of $L_2+L_3$, we get that
\begin{equation}\label{L2_ca1}
\begin{split}
Re(L_2+L_3)=&Re\left( i\sum_{j=1}^{N}\left(\left.-\int_{I_j}|\partial_xu_j|^2dx+\partial_xu_j\overline{u}_j\right]_{\partial I_j}\right)\right)\\
&+Re\left(i\sum_{j=1}^{N}\left(-\left.\left.\int_{I_j}|\partial^2_xu_j|^2dx+\partial^2_xu_j\partial_x\overline{u}_j\right]_{\partial I_j}-\partial^3_xu_j\overline{u}_j\right]_{\partial I_j}\right)\right)\\
=&Re\left.i\left(\partial_xu_1\overline{u}_1+\partial^2_xu_1\partial_x\overline{u}_1-\partial^3_xu_1\overline{u}_1\right)\right]_{-l_1}^0\\
&+Re\left.i\sum_{j=2}^{N}\left(\partial_xu_j\overline{u}_j+\partial^2_xu_j\partial_x\overline{u}_j-\partial^3_xu_j\overline{u}_j\right)\right]_0^{l_j}.
\end{split}
\end{equation}
By using the boundary conditions \eqref{adj_bound1aa}, we have
\begin{equation}\label{L2_ca2}
\begin{split}
Re(L_2+L_3)=&Re \,\, i\left(\partial_xu_1(0)\overline{u}_1(0)+\partial^2_xu_1(0)\partial\overline{u}_1(0)-\partial^3u_1(0)\overline{u}_1(0)\right.\\
&+Re\left.i\sum_{j=2}^{N}\left(-\partial_xu_j(0)\overline{u}_j(0)-\partial^2_xu_j(0)\partial_x\overline{u}_j(0)+\partial^3_xu_j(0)\overline{u}_j(0)\right)\right.=0.
\end{split}
\end{equation}
Thus, replacing \eqref{L2_ca2} in \eqref{L2_caa}, \eqref{L2_c} holds.

We will prove \eqref{H12_c}. Multiplying \eqref{adj_graph_1aa}, with $f=0$, by $\overline{\partial_tu}$, integrating on $\mathcal{G}$ and taking the real part give us
\begin{equation}\label{L2_ca2a}
Re\left(i\int_{\mathcal{G}}|\partial_tu|^2dx\right)+Re\left(\int_{\mathcal{G}}\partial^2_xu\overline{\partial_tu}dx\right)-Re\left(\int_{\mathcal{G}}\partial^4_xu\overline{\partial_tu}dx\right)=0.
\end{equation}
Integrating \eqref{L2_ca2a} by parts on $\mathcal{G}$ and using the boundary conditions \eqref{adj_bound1aa}, yields that
\begin{equation}\label{L2_ca2aa}
\begin{split}
\frac{1}{2}\frac{\partial}{\partial_t}\int_{\mathcal{G}}&\left(|\partial_xu|^2+|\partial^2_xu|^2\right)dx=-Re\left[\partial^2_xu_1(0)\partial_x\partial_t\overline{u}_1(0)+\partial_xu_1(0)\partial_t\overline{u}_1(0)-\partial^3_xu_1(0)\partial_t\overline{u}_1(0)\right]\\
&+\sum_{j=2}^{N}\left(-\partial^2_xu_j(0)\partial_x\partial_t\overline{u}_j(0)-\partial_xu_j(0)\partial_t\overline{u}_j(0)+\partial^3_xu_j(0)\partial_t\overline{u}_j(0)\right).
\end{split}
\end{equation}
The boundary condition give us that the right hand side of \eqref{L2_ca2aa} is zero, that is
$$\frac{1}{2}\frac{\partial}{\partial_t}\int_{\mathcal{G}}\left(|\partial_xu|^2+|\partial^2_xu|^2\right)dx=0,$$
which implies \eqref{H12_c}. Finally, adding \eqref{L2_c} and \eqref{H12_c}, we have \eqref{H2_c}.
\end{proof}

\subsection*{Acknowledgments} 
Capistrano–Filho was supported by CNPq 408181/2018-4, CAPES-PRINT 88881.311964/2018-01, CAPES-MATHAMSUD 88881.520205/2020-01, MATHAMSUD 21- MATH-03 and Propesqi (UFPE). Cavalcante was supported by CAPES-MATHAMSUD 88887.368708/2019-00. Gallego was supported by MATHAMSUD 21-MATH-03  and the 100.000 Strong in the Americas Innovation Fund. This work was carried out during visits of the authors to the Universidade Federal de Alagoas, Universidade Federal de Pernambuco and Universidad Nacional de Manizales. The authors would like to thank the universities for its hospitality.

\end{document}